\documentclass[11pt,ps2pdf]{article}

\usepackage{amsthm}
\usepackage{amssymb}
\usepackage{color}
\usepackage{tikz}
\usepackage{pgf}
\usepackage{float}

\newtheorem{lem}{Lemma}
\newtheorem{prop}{Proposition}
\newtheorem{thm}{Theorem}
\newtheorem{cor}{Corollary}
\newtheorem{exm}{Example}
\newtheorem{rmk}{Remark}

\newcommand{\pr}{{\rm Pr}}

\usepackage{amsmath}	

\newcommand{\eunump}[2]{ \left\langle
 \begin{array}{c}
  #1\\#2
 \end{array}
 \right\rangle \hspace{-2mm}\raisebox{-3mm}{\tiny $p$}
}

\begin{document}

\begin{center}
 {\Large A  generalization of carries processes and Eulerian numbers}

 \bigskip
 {\large Fumihiko Nakano\footnote{Department of Mathematics, Gakushuin University, 1-5-1, Mejiro, Toshima-ku,
Tokyo, 171-8588, Japan. e-mail : fumihiko@math.gakushuin.ac.jp} 
 and Taizo Sadahiro\footnote{
 Department of Computer Science, Tsuda Colledge, Tokyo, Japan. e-mail :
 sadahiro@tsuda.ac.jp}
 }
\end{center}

\begin{abstract}
 We study a generalization of Holte's amazing matrix,
 the
 transition probability matrix of the Markov chains of
 the 'carries' in a {\em non-standard} numeration system.
 The stationary distributions are explicitly described by
 the numbers which can be regarded as a
 generalization of the Eulerian numbers and
 the MacMahon numbers.
 We also show that similar properties hold
 even for the numeration systems with the negative bases.
\end{abstract}



\section{Introduction and statements of results}
The transition probability matrix so-called 'amazing matrix' of
the Markov chain of the 'carries' has very nice properties \cite{Holte},
and has unexpected connection to
the Markov chains of riffle shuffles \cite{DiaconisFulman,DiaconisFulman2}.
Diaconis and Fulman \cite{DiaconisFulman2} studies
a variant of the carries process, type $B$ carries process.
Novelli and Thibon studies the carries process in terms of
noncommutative symmetric functions \cite{NovelliThibon}.
This paper studies a generalization of the carries process which 
includes Diaconis and Fulman's type $B$ carries process as a
special case.
We study the transition probability matrices of
the Markov chains of the carries in the numeration systems with non-standard digit sets.
We show that the matrices have the eigenvectors which can be
perfectly described by a generalization of Eulerian numbers and
the MacMahon numbers \cite{oeismacmahon,MacMahon,ChowGessel,DiaconisFulman2}.
We also show that similar properties hold even for
the numeration systems with negative bases.

\subsection{Numeration system}

Throughout the paper, $b$ denotes a positive integer and 
${\mathcal D} = \{d, d + 1, \ldots, d + b - 1\}$ denotes a
set of integers containing $0$. Therefore, $-b < d < b$.
Then, we have a numeration system $(b,{\mathcal D})$:
Suppose that an integer $x$ has a representation of the form,
\begin{equation}\label{eq:rep}
 x = (x_k x_{k-1}\cdots x_0)_b \stackrel{def}{=} x_0 + x_1b + x_2b^2 + \cdots + x_kb^k,
 ~~~
 x_0, x_1, \ldots ,x_k\in {\mathcal D}, x_k\neq 0.
\end{equation}
Then, it can be easily shown that 
this representation is uniquely determined for $x$ and
\[
 \left\{
  (x_k x_{k-1}\cdots x_0)_b
  \,\bigg|\, k\geq 0,
 x_0, x_1, \ldots ,x_k\in {\mathcal D}  
 \right\}
 =
 \begin{cases}
  {\mathbb Z} & d \neq 0,-b + 1,\\
  {\mathbb N} & d = 0,\\
  -{\mathbb N} & d = -b + 1.\\
 \end{cases}
\]
is closed under the addition,
where ${\mathbb N}$ denotes the set of
non-negative integers.

\subsection{Carries process}
Let $\{X_{i,j}\}_{1\leq i\leq n, j\geq 0}$  be the set of 
independent random variables each of which 
is distributed uniformly over ${\mathcal D}$.
Define the two stochastic processes $(A_0, A_1, A_2,\ldots)$, 
and $(C_0, C_1, C_2, \ldots)$ in
the following way:
$C_0 = 0$  with probability one.
$(A_i)_{i\geq 0}$ is a sequence of ${\mathcal D}$-valued random variables satisfying
\[
 A_i \equiv  C_i + X_{1,i} + \cdots + X_{n,i}\pmod{b},~~~ i = 0,1,2,\ldots.
\]
and
\[
 C_i = \frac{C_{i - 1} + X_{1,i-1} + \cdots + X_{n,i-1} - A_{i-1}}{b},
 ~~~
 i = 1,2,3,\ldots.
\]
(See Figure $\ref{fig:carryprocess}$.)
It is obvious that $(C_0, C_1, C_2, \ldots )$ is a Markov process,
which we call  
the {\em carries process} with $n$ summands or simply $n$-{\em carry process}
over $(b,{\mathcal D})$.

\begin{figure}[H]
  \[
  \begin{array}{cccccccc}
   & \cdots  & C_4 & C_3 & C_2 & C_1 & C_0\\
   \\
    &\cdots     & X_{1,4} & X_{1,3} & X_{1,2} & X_{1,1} & X_{1,0}\\
    &\cdots     & X_{2,4} & X_{2,3} & X_{2,2} & X_{2,1} & X_{2,0}\\
    &\cdots     & X_{3,4} & X_{3,3} & X_{3,2} & X_{3,1} & X_{3,0}\\
    &\cdots     & \vdots & \vdots & \vdots & \vdots & \vdots\\
   +)&\cdots     & X_{n,4} & X_{n,3} & X_{n,2} & X_{n,1} & X_{n,0}\\
   \hline
    &\cdots & A_4  & A_3 & A_2 & A_1 & A_0
  \end{array}
  \]
 \caption{Carries process}
 \label{fig:carryprocess}
\end{figure}

%

\subsection{A generalization of Eulerian numbers}
Let $p\geq 1$ be a real number and  $n$ a positive integer.
Then we define an array of numbers $v_{i,j}^{(p)}(n)$ for
$i = 0,1,\ldots,n$ and $j = 0, 1,\ldots, n+1$ by
\begin{equation}\label{eq:gedef}
 v_{i,j}^{(p)}(n) =
 \sum_{r = 0}^j
 (-1)^r
 {n + 1\choose r} 
 [p(j - r) + 1]^{n-i},
\end{equation}
and define $v_{i,-1}^{(p)}(n) = 0$.
We denote
\[
 \eunump{n}{j} = v_{0,j}^{(p)}(n),
\]
which can be regarded as a generalization
of the Eulerian numbers.
In fact, $\left\{\eunump{n}{j}\right\}$ forms the array of the 
ordinary Eulerian numbers when $p = 1$, and  
MacMahon numbers \cite{oeismacmahon,MacMahon,ChowGessel,DiaconisFulman2} when $p = 2$.

\subsection{Statement of the result}

Throughout the paper, $\Omega = \Omega_n(b,{\mathcal D})$ denotes the
state space of the $n$-carry process over $(b,{\mathcal D})$,
that is, the set of possible values of carries, and
$p_{i,j}$ denotes the 
transition probability $\pr(C_{i + 1} = j \,|\, C_{i} =i)$
for $i,j\in\Omega_n(b,{\mathcal D})$.
For computational convenience, it is desirable for the transition probability 
matrix to have indices starting from $0$.
We define $\tilde{p}_{i,j} = p_{i + s, j + s}$, where
$s$ is the minimal element of $\Omega_n(b,{\mathcal D})$.
Then, we define the matrix $P$ by
\begin{equation}
 \label{eq:amazingmat}
 P = \left(\tilde{p}_{i,j}\right)_{0\leq i,j\leq \#\Omega -1},
\end{equation}
where $\#\Omega$ denotes the size of the state space,
which is explicitly computed in Lemma $\ref{lem:statespace}$.
$P$ is the central object of this paper.
\begin{rmk}
 As we will show in the later sections,
 this matrix $P$ which we regard as a generalization of
 Holte's amazing matrix is determined only by $b, n$ and $p$,
 and therefore these amazing matrices with same $n$ and $p$ form a
 commutative family.
 Holte's amazing matrix corresponds to the case when $p = 1$ and
 Diaconis and Fulman's type B carries process corresponds to
 the case when $p = 2$.
\end{rmk}

\begin{thm}\label{thm:main}
 Let $\Omega = \Omega_n(b,{\mathcal D})$ 
 be the state space of the $n$-carry 
 process over $(b,{\mathcal D})$, and
 $m = \#\Omega_n(b,{\mathcal D})$.
 let $p$ be defined by
 \begin{equation}\label{eq:mainth}
  p = 
 \begin{cases}
   \frac{1}{\{(n - 1)(-l)\}} & (n - 1)l \not\in {\mathbb Z} 
  ,\\
  1 &  (n - 1)l \in {\mathbb Z}
  ,
 \end{cases}
 \end{equation}
 where $l = d/(b - 1)$ and $\{x\} = x - \lfloor x\rfloor$,
 and let $V = \left(v_{i,j}^{(p)}(n)\right)_{0\leq i,j\leq m - 1}$.
 Then, 
 the transition probability matrix $P$
 is diagonalized by $V$:
 \[
  VPV^{-1} = {\rm diag}\left(1,b^{-1},\ldots, b^{-(m - 1)}\right).
 \]
 In particular, by {\rm Lemma} $\ref{lem:sumeuler}$, the probability vector $\pi$ of the stationary distribution 
 of the carries process is
 \[
 \pi = (\pi(s),\pi(s + 1),\ldots, \pi(s + m - 1)) = 
 \frac{1}{p^n n!}
 \left(\eunump{n}{0},\ldots,\eunump{n}{m - 1}\right)
 \]
\end{thm}

\begin{rmk}
 It is remarkable that our amazing matrix has the
 eigenvalues of the same form $1,1/b,1/b^2,\ldots$ as those of
 Holte's amazing matrix.
\end{rmk}

\begin{cor}\label{cor:sumeuler}
 Let $S_n$ be the sum of $n$ independent random variables
 each of which is distributed uniformly over the unit interval $[0,1]$.
 Then, for all positive real numbers $p \geq 1$ and integers $k$,
 the probability of $S_n$ being in the interval $\frac{1}{p} + [k-1,k]$
 is
 \begin{equation}
  \label{eq:sumunifiid}
 \pr\left( S_n \in \frac{1}{p} + [k -1, k]\right) =
 \frac{1}{p^nn!}\eunump{n}{k}.
 \end{equation}
\end{cor}
%
\begin{rmk}
 This corollary can be derived directly from the
 formula of the distribution of sums of independent
 uniform random variables in {\rm  \cite{Feller}}, and 
 it is shown for the case $p = 2$ in {\rm  \cite{DiaconisFulman2}}.
\end{rmk}

\begin{exm}
Let $p \geq 1$ be a real number.
As will be shown in the later section, the array of generalized
Eulerian numbers satisfies the following recursive relations
\[
 \eunump{n + 1}{k} = (pk + 1)\eunump{n}{k} + (p(n + 1 - k) - 1)\eunump{n}{k - 1}.
\]
and the boundary conditions
\[
 \eunump{n}{0} = 1,\mbox{~~ and ~~} \eunump{n}{k} = 0 \mbox{~ for ~} k\ > n.
\]
(See Figure $\ref{fig:gentri}$.)

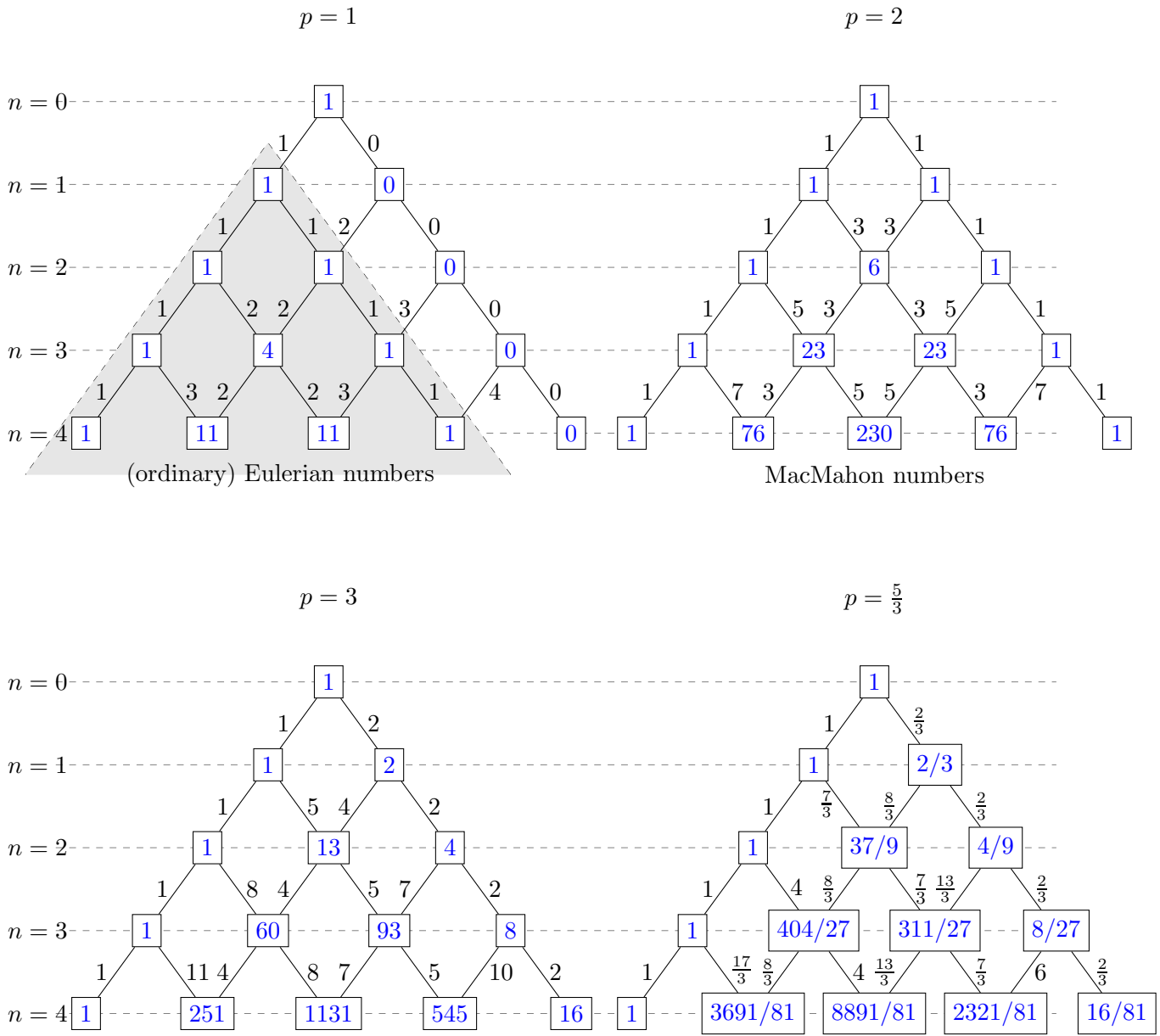
\begin{figure}[H]
 \begin{center}
      \begin{tikzpicture}[yscale = 1.3,xscale = 0.95]
     \draw (0,1) node {$p = 1$};
     \foreach \n in {0,1,2,3,4}{
     \draw [very thin, gray,dashed] (-4.5,-\n) -- (12,-\n);
     }
     \draw [very thin,dashed,fill = gray!20] (3.,-4.5) -- (-1,-0.5) -- (-5,-4.5);
     \fill [fill = gray!20] (-1,-0.5) -- (-5,-4.5) -- (3.,-4.5) -- cycle;
     \draw (-.8,-4.5) node {(ordinary) Eulerian numbers};
      \node [fill = white, draw](a10)at (0,0) {\color{blue}$1$};

      \node [fill = white, draw](a20)at (-1,-1) {\color{blue}$1$};
      \node [fill = white, draw](a21)at (1,-1) {\color{blue}$0$};

      \node [fill = white, draw](a30)at (-2,-2) {\color{blue}$1$};
      \node [fill = white, draw](a31)at (0,-2) {\color{blue}$1$};
      \node [fill = white, draw](a32)at (2,-2) {\color{blue}$0$};

      \node [fill = white, draw](a40)at (-3,-3) {\color{blue}$1$};
      \node [fill = white, draw](a41)at (-1,-3) {\color{blue}$4$};
      \node [fill = white, draw](a42)at (1,-3) {\color{blue}$1$};
      \node [fill = white, draw](a43)at (3,-3) {\color{blue}$0$};

      \node [fill = white, draw](a50)at (-4,-4) {\color{blue}$1$};
      \node [fill = white, draw](a51)at (-2,-4) {\color{blue}$11$};
      \node [fill = white, draw](a52)at (0,-4) {\color{blue}$11$};
      \node [fill = white, draw](a53)at (2,-4) {\color{blue}$1$};
      \node [fill = white, draw](a54)at (4,-4) {\color{blue}$0$};

      \path (a10) edge node [left] {$1$}(a20);
      \path (a10) edge node [right] {$0$} (a21);

      \path (a20) edge node [left] {$1$}(a30);
      \path (a20) edge node [right] {$1$} (a31);
      \path (a21) edge node [left] {$2$}(a31);
      \path (a21) edge node [right] {$0$} (a32);

      \path (a30) edge node [left] {$1$}(a40);
      \path (a30) edge node [right] {$2$} (a41);
      \path (a31) edge node [left] {$2$}(a41);
      \path (a31) edge node [right] {$1$} (a42);
      \path (a32) edge node [left] {$3$}(a42);
      \path (a32) edge node [right] {$0$} (a43);

      \path (a40) edge node [left] {$1$}(a50);
      \path (a40) edge node [right] {$3$} (a51);
      \path (a41) edge node [left] {$2$}(a51);
      \path (a41) edge node [right] {$2$} (a52);
      \path (a42) edge node [left] {$3$}(a52);
      \path (a42) edge node [right] {$1$} (a53);
      \path (a43) edge node [right] {$4$} (a53);
      \path (a43) edge node [right] {$0$} (a54);

      \draw rectangle (0,0) + (1,1);

     \foreach \n in {0,1,2,3,4}{
     \draw (-4.8,-\n) node {$n = \n$};
     }

     \begin{scope}[xshift = 9cm]
      \draw (0,1) node {$p = 2$};
      \draw (-.0,-4.5) node {MacMahon numbers};
      \node [draw,fill = white] (a10)at (0,0) {\color{blue}$1$};

      \node [draw,fill = white] (a20)at (-1,-1) {\color{blue}$1$};
      \node [draw,fill = white] (a21)at (1,-1) {\color{blue}$1$};

      \node [draw,fill = white] (a30)at (-2,-2) {\color{blue}$1$};
      \node [draw,fill = white] (a31)at (0,-2) {\color{blue}$6$};
      \node [draw,fill = white] (a32)at (2,-2) {\color{blue}$1$};

      \node [draw,fill = white] (a40)at (-3,-3) {\color{blue}$1$};
      \node [draw,fill = white] (a41)at (-1,-3) {\color{blue}$23$};
      \node [draw,fill = white] (a42)at (1,-3) {\color{blue}$23$};
      \node [draw,fill = white] (a43)at (3,-3) {\color{blue}$1$};

      \node [draw,fill = white] (a50)at (-4,-4) {\color{blue}$1$};
      \node [draw,fill = white] (a51)at (-2,-4) {\color{blue}$76$};
      \node [draw,fill = white] (a52)at (0,-4) {\color{blue}$230$};
      \node [draw,fill = white] (a53)at (2,-4) {\color{blue}$76$};
      \node [draw,fill = white] (a54)at (4,-4) {\color{blue}$1$};

      \path (a10) edge node [left] {$1$}(a20);
      \path (a10) edge node [right] {$1$} (a21);

      \path (a20) edge node [left] {$1$}(a30);
      \path (a20) edge node [right] {$3$} (a31);
      \path (a21) edge node [left] {$3$}(a31);
      \path (a21) edge node [right] {$1$} (a32);

      \path (a30) edge node [left] {$1$}(a40);
      \path (a30) edge node [right] {$5$} (a41);
      \path (a31) edge node [left] {$3$}(a41);
      \path (a31) edge node [right] {$3$} (a42);
      \path (a32) edge node [left] {$5$}(a42);
      \path (a32) edge node [right] {$1$} (a43);

      \path (a40) edge node [left] {$1$}(a50);
      \path (a40) edge node [right] {$7$} (a51);
      \path (a41) edge node [left] {$3$}(a51);
      \path (a41) edge node [right] {$5$} (a52);
      \path (a42) edge node [left] {$5$}(a52);
      \path (a42) edge node [right] {$3$} (a53);
      \path (a43) edge node [right] {$7$} (a53);
      \path (a43) edge node [right] {$1$} (a54);

      \draw rectangle (0,0) + (1,1);
     \end{scope}

     \begin{scope}[yshift = -7cm]
      \draw (0,1) node {$p = 3$};
      \foreach \n in {0,1,2,3,4}{
      \draw [very thin, gray,dashed] (-4.5,-\n) -- (12,-\n);
      }
      \node [draw,fill = white] (a10)at (0,0) {\color{blue}$1$};

      \node [draw,fill = white] (a20)at (-1,-1) {\color{blue}$1$};
      \node [draw,fill = white] (a21)at (1,-1) {\color{blue}$2$};

      \node [draw,fill = white] (a30)at (-2,-2) {\color{blue}$1$};
      \node [draw,fill = white] (a31)at (0,-2) {\color{blue}$13$};
      \node [draw,fill = white] (a32)at (2,-2) {\color{blue}$4$};

      \node [draw,fill = white] (a40)at (-3,-3) {\color{blue}$1$};
      \node [draw,fill = white] (a41)at (-1,-3) {\color{blue}$60$};
      \node [draw,fill = white] (a42)at (1,-3) {\color{blue}$93$};
      \node [draw,fill = white] (a43)at (3,-3) {\color{blue}$8$};

      \node [draw,fill = white] (a50)at (-4,-4) {\color{blue}$1$};
      \node [draw,fill = white] (a51)at (-2,-4) {\color{blue}$251$};
      \node [draw,fill = white] (a52)at (0,-4) {\color{blue}$1131$};
      \node [draw,fill = white] (a53)at (2,-4) {\color{blue}$545$};
      \node [draw,fill = white] (a54)at (4,-4) {\color{blue}$16$};

      \path (a10) edge node [left] {$1$}(a20);
      \path (a10) edge node [right] {$2$} (a21);

      \path (a20) edge node [left] {$1$}(a30);
      \path (a20) edge node [right] {$5$} (a31);
      \path (a21) edge node [left] {$4$}(a31);
      \path (a21) edge node [right] {$2$} (a32);

      \path (a30) edge node [left] {$1$}(a40);
      \path (a30) edge node [right] {$8$} (a41);
      \path (a31) edge node [left] {$4$}(a41);
      \path (a31) edge node [right] {$5$} (a42);
      \path (a32) edge node [left] {$7$}(a42);
      \path (a32) edge node [right] {$2$} (a43);

      \path (a40) edge node [left] {$1$}(a50);
      \path (a40) edge node [right] {$11$} (a51);
      \path (a41) edge node [left] {$4$}(a51);
      \path (a41) edge node [right] {$8$} (a52);
      \path (a42) edge node [left] {$7$}(a52);
      \path (a42) edge node [right] {$5$} (a53);
      \path (a43) edge node [right] {$10$} (a53);
      \path (a43) edge node [right] {$2$} (a54);

      \draw rectangle (0,0) + (1,1);

     \foreach \n in {0,1,2,3,4}{
     \draw (-4.8,-\n) node {$n = \n$};
     }

      \begin{scope}[xshift = 9cm]
       \draw (0,1) node {$p = \frac{5}{3}$};
       \node [draw,fill = white] (a10)at (0,0) {\color{blue}$1$};

       \node [draw,fill = white]  (a20)at (-1,-1) {\color{blue}$1$};
       \node [draw,fill = white]  (a21)at (1,-1) {\color{blue}$2/3$};

       \node [draw,fill = white] (a30)at (-2,-2) {\color{blue}$1$};
       \node [draw,fill = white] (a31)at (0,-2) {\color{blue}$37/9$};
       \node [draw,fill = white] (a32)at (2,-2) {\color{blue}$4/9$};

       \node [draw,fill = white] (a40)at (-3,-3) {\color{blue}$1$};
       \node [draw,fill = white] (a41)at (-1,-3) {\color{blue}$404/27$};
       \node [draw,fill = white] (a42)at (1,-3) {\color{blue}$311/27$};
       \node [draw,fill = white] (a43)at (3,-3) {\color{blue}$8/27$};

       \node [draw,fill = white] (a50)at (-4,-4) {\color{blue}$1$};
       \node [draw,fill = white] (a51)at (-2,-4) {\color{blue}$3691/81$};
       \node [draw,fill = white] (a52)at (0,-4) {\color{blue}$8891/81$};
       \node [draw,fill = white] (a53)at (2,-4) {\color{blue}$2321/81$};
       \node [draw,fill = white] (a54)at (4,-4) {\color{blue}$16/81$};

       \path (a10) edge node [left] {$1$}(a20);
       \path (a10) edge node [right] {\small$\frac{2}{3}$} (a21);

       \path (a20) edge node [left] {$1$}(a30);
       \path (a20) edge node [left] {\small$\frac{7}{3}$} (a31);
       \path (a21) edge node [left] {\small$\frac{8}{3}$}(a31);
       \path (a21) edge node [right] {\small$\frac{2}{3}$} (a32);

       \path (a30) edge node [left] {$1$}(a40);
       \path (a30) edge node [right] {$4$} (a41);
       \path (a31) edge node [left] {\small$\frac{8}{3}$}(a41);
       \path (a31) edge node [right] {$\frac{7}{3}$} (a42);
       \path (a32) edge node [left] {\small$\frac{13}{3}$}(a42);
       \path (a32) edge node [right] {\small$\frac{2}{3}$} (a43);

       \path (a40) edge node [left] {$1$}(a50);
       \path (a40) edge node [right] {\small$\frac{17}{3}$} (a51);
       \path (a41) edge node [left] {\small$\frac{8}{3}$}(a51);
       \path (a41) edge node [right] {$4$} (a52);
       \path (a42) edge node [left] {\small$\frac{13}{3}$}(a52);
       \path (a42) edge node [right]  {\small$\frac{7}{3}$} (a53);
       \path (a43) edge node [right] {$6$} (a53);
       \path (a43) edge node [right] {\small$\frac{2}{3}$} (a54);

       \draw rectangle (0,0) + (1,1);
      \end{scope}
     \end{scope}

  \end{tikzpicture}

  \caption{Arrays of generalized Eulerian numbers
  for $p=1$ (upper left),   $p = 2$ (upper right), $p = 3$ (lower left), and $p = \frac{5}{3}$ (lower right)}
  \label{fig:gentri}
 \end{center}
\end{figure}

 The probability density function of $S_3$ described in
 Theorem $\ref{cor:sumeuler}$ is
 \[
  f(x) = 
  \begin{cases}
   x^2/2 &  \mbox{~~if~~} 0\leq x < 1,\\
   -\left(x - \frac{3}{2}\right)^2 + \frac{3}{4} & \mbox{~~if~~} 1\leq x < 2,\\
   (x - 3)^2/2 & \mbox{~~if~~} 2\leq x \leq 3,\\
   0 & \mbox{ otherwise}.
  \end{cases}
 \]

 
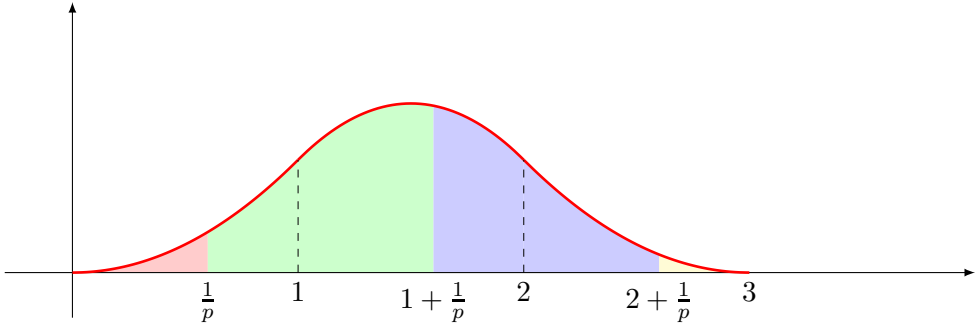
\begin{figure}[H]
  \begin{center}
   \begin{tikzpicture}[scale = 3]

    \fill [fill = red!20, domain = 0:0.6, smooth] plot (\x, {\x*\x/2}) -- (0.6,0) -- (0,0) -- cycle;
    \fill [fill = green!20, domain = 0.6:1, smooth] plot (\x, {\x*\x/2}) -- (1,0) -- (0.6,0) -- cycle;
    \fill [fill = green!20, domain = 1:1.6, smooth] plot (\x, {(1.5 - \x)*(\x - 1.5) + 0.75}) -- (1.6,0) -- (1,0) -- cycle;

    \fill [fill = blue!20, domain = 1.6:2, smooth] plot (\x, {(1.5 - \x)*(\x - 1.5) + 0.75}) -- (2,0) -- (1.6,0) -- cycle;
    \fill [fill = blue!20, domain = 2:2.6, smooth] plot (\x, {(\x - 3)*(\x - 3)/2}) -- (2.6,0) -- (2,0) -- cycle;
    \fill [fill = yellow!20, domain = 2.6:3, smooth] plot (\x, {(\x - 3)*(\x - 3)/2}) -- (3,0) -- (2.6,0) -- cycle;

    \draw (0.6,0)  node[below] {$\frac{1}{p}$};
    \draw (1.6,0)  node[below] {$1 + \frac{1}{p}$};
    \draw (2.6,0)  node[below] {$2 + \frac{1}{p}$};
    \draw (1,0) node[below] {$1$};
    \draw (2,0) node[below] {$2$};
    \draw (3,0) node[below] {$3$};

    \pgfsetarrowsend{latex}
    \draw (-0.3,0) -- (4,0);
    \draw (0,-0.2) -- (0,1.2);
    \pgfsetarrowsend{}

    \draw [color = red, domain = 0:1, smooth, line width = 1] plot (\x, {\x*\x/2});
    \draw [color = red, domain = 1:2, smooth, line width = 1] plot (\x, {(1.5 - \x)*(\x - 1.5) + 0.75});
    \draw [color = red, domain = 2:3, smooth, line width = 1] plot (\x, {(\x - 3)*(\x - 3)/2});
    \draw [line width = 0.01,dashed] (1,0) -- (1,0.5);
    \draw [line width = 0.01,dashed] (2,0) -- (2,0.5);
   \end{tikzpicture}
  \end{center}
 \end{figure}

 \begin{table}[H]
  \begin{center}
   \begin{tabular}{c|cccc}
    $p$ & $\pr(S_3\in 1/p + [-1,0] )$ & $\pr(S_3\in 1/p + [0,1])$ & $\pr(S_3\in 1/p + [1,2])$ & 
    $\pr(S_3\in 1/p + [2,3])$\\
    \hline
    \\
    1 & $\frac{1}{6}$ & $\frac{4}{6}$ & $\frac{1}{6}$ & 0\\
    \\
    2 & $\frac{1}{48}$ & $\frac{23}{48}$ & $\frac{23}{48}$ & $\frac{1}{48}$\\
    \\
    3 & $\frac{1}{162}$ & $\frac{60}{162}$ & $\frac{93}{162}$ & $\frac{8}{162}$  \\
    \\
    ${5}/{3}$ & $\frac{27}{750}$ & $\frac{404}{750}$ & $\frac{311}{750}$ & $\frac{8}{750}$
   \end{tabular}
  \end{center}
 \end{table}
 The probability vectors appear in the third rows
 of the triangles of the generalized Eulerian numbers
 (Figure $\ref{fig:gentri}$).
\end{exm}

\section{Proof}

\subsection{State space and transition probability}

 \begin{lem}\label{lem:statespace}
  Let $\Omega  = \Omega_n(b,{\mathcal D})$ be the state space of the
  $n$-carry process over the numeration system 
  $(b,{\mathcal D})$.
  Then, $\Omega = \{s,s+1, \ldots,t\}$ with
  \[
   s = -\lceil(n - 1)(- l)\rceil = \lfloor(n - 1)l\rfloor,
   ~~
   t = \lceil (n - 1)(l + 1) \rceil,
  \]
  where $l = d/(b - 1)$.
  Therefore, the size of the state space $\Omega$ is
  \[
   \#\Omega =
  \begin{cases}
   n + 1 & (n - 1)l \not\in {\mathbb Z},\\
   n &  (n - 1)l \in {\mathbb Z}.
  \end{cases}
  \]
 \end{lem}

\begin{proof}
 Suppose that we add $n$ numbers,
 \[
 (x_{1,\nu}x_{1,\nu - 1}\cdots x_{1,1}x_{1,0})_b,
 (x_{2,\nu}x_{2,\nu - 1}\cdots x_{2,1}x_{2,0})_b,\ldots,
 (x_{n,\nu}x_{n,\nu - 1}\cdots x_{n,1}x_{n,0})_b,
 \]
 and get the sum $(a_{\nu + 1},a_{\nu},\ldots, a_1, a_0)_b$ and the carries $c_0, c_1, c_2, \ldots$.
 That is, $c_0 = 0$ and
 \[
  c_{i + 1} = \frac{c_i + x_{1,i} + \cdots + x_{n,i} - a_i}{b},
 \]
 where $a_i \in {\mathcal D}$ and $a_i \equiv c_i + x_{1,i} + \cdots + x_{n,i}\pmod{ b}$.
  Let $F$ be defined by
  \[
   F = \{x_1b^{-1} + x_2b^{-2} + \cdots + x_mb^{-m}\,|\,
 x_i \in {\mathcal D}_b, m > 0 \in {\mathbb Z}\}.
  \]
  Then, $F$ is a dense subset of the
 interval $(l, l+1)$, where $l = d/(b - 1)$.
  \begin{eqnarray*}
   c_i = c & \Longleftrightarrow &
    (x_{1,i-1}\cdots x_{1,0})_{b} + \cdots + (x_{n,i-1}\cdots x_{n,0})_{b}
    = cb^i + (a_{i-1}\cdots a_0)_{b}\\
   & \Longleftrightarrow &
    \frac{(x_{1,i-1}\cdots x_{1,0})_{b}}{b^i} + \cdots + 
    \frac{(x_{n,i-1}\cdots x_{n,0})_{b}}{b^i}
    = c + \frac{(a_{i-1}\cdots a_0)_{b}}{b^i}\\
  \end{eqnarray*}
  Since $\frac{(x^1_{i-1}\cdots x^1_{0})_{b}}{b^i}, \ldots,  
  \frac{(x^n_{i-1}\cdots x^n_{0})_{b}}{b^i},
  \frac{(a_{i-1}\cdots a_0)_{b}}{b^i} \in F$,
  we have $c\in nF - F \subset ((n - 1)l - 1, (n - 1)(l + 1) + 1)$.
 Conversely, if $c \in ((n - 1)l - 1, (n - 1)(l + 1) + 1)\cap {\mathbb Z}$,
 then  $c + F \subset nF$.
  Therefore $s$ is the smallest integer strictly greater than
  $(n - 1)l - 1$ and $t$ is the greatest integer
  strictly smaller than $(n - 1)(l + 1) + 1$, that is,
 \[
 s = -\left\lceil(n - 1)(-l)\right\rceil,~~~
 t = \left\lceil(n - 1)(l + 1)\right\rceil.
 \]
\end{proof}

 \begin{thm}\label{thm:transp}
  \[
   p_{i,j}
  =
  \frac{1}{b^n}
  \sum_{k = 0}^{j - \left\lfloor\frac{d(n - 1) + i }{b}\right\rfloor}
  (-1)^k{n + 1 \choose k}
  {n + b(j + 1 - k) - d(n - 1) - i - 1 \choose n},
  ~~~~~i,j\in \Omega_n(b,{\mathcal D}).
  \]
 \end{thm}
\begin{proof}
 This proof is essentially the same as that of Holte \cite{Holte} for the case when ${\mathcal D} = \{0,1,\ldots,b - 1\}$.
 $p_{i,j}$ is the probability of  $C_{k + 1} = j$
 under  $C_k = i$ for some $k > 0$.
 $C_{k + 1} = j$ and $C_{k} = i$ implies
 there exists a number $a\in{\mathcal D}$, such that,
 \[
 j = \frac{i + X_{1,k} + X_{2,k} + \cdots + X_{n,k} - a}{b}.
 \]
 We count the number $N$ of the solutions $(x_1, x_2, \ldots, x_n, a)\in{\mathcal D}^{n + 1}$
 of the equation
 \[
 bj + a = i + x_1 + x_2 + \cdots + x_n.
 \]
 This is equal to the number of solutions $(x_1, x_2, \ldots, x_n, y)\in{\mathcal D}^{n}\times\{0,1,\ldots,b - 1\}$
 for the equation
 \[
  bj + d + b - 1 - i = x_1 + x_2 + \cdots + x_n + y.
 \]
 By adding $d$ to the both sides, $N$ is equal to
 the number of solutions $(x_1, \ldots, x_n, z)\in{\mathcal D}^{n + 1}$
 for 
 \[
  b(j + 1) + 2d - 1 - i = x_1 + x_2 + \cdots + x_n + z.
 \]
 Thus, $N$ is the coefficient of $x^{b(j + 1) + 2d - 1 - i}$ in
 $  (x^{d} + x^{d + 1} + \cdots + x^{d + b - 1})^{n + 1}$.
 Since
 \begin{eqnarray*}
  (x^{d} + x^{d + 1} + \cdots + x^{d + b - 1})^{n + 1} 
   & = &
   \left\{
    x^{d}(1 + x + \cdots + x^{b - 1})
       \right\}^{n + 1}\\
  & = & 
   x^{d(n+1)}
   \left(
    \frac{1 - x^{b}}{1 -x}
   \right)^{n + 1}\\
   & = &
   x^{d(n+1)}
   \left(
   \sum_{k = 0}^{n+1}
   (-1)^k{n + 1 \choose k}
   x^{bk}
   \right)
   \left(
    \sum_{r = 0}^\infty
    {n + r\choose n}x^r
   \right),
 \end{eqnarray*}
 we have
 \begin{eqnarray*}
  N & = & 
   \sum_{\tiny
 \begin{array}{c}
  r,k\geq 0\\
  d(n + 1) + bk + r = b(j + 1) + 2d - 1 - i
 \end{array}
 }
 (-1)^k{n + 1 \choose k}
 {n + r\choose n}x^r\\
  & = &
   \sum_{\tiny
   \begin{array}{c}
    k\geq 0\\
    b(j + 1) + 2d - 1 - i - d(n + 1) - bk \geq 0
   \end{array}
   }
   (-1)^k{n + 1 \choose k}
   {n + b(j + 1) + 2d - 1 - i - d(n + 1) - bk\choose n}x^r\\
  & = &
   \sum_{k = 0}
    ^{
    j + 1 + \left\lfloor - \frac{1 + i + d(n - 1)}{b}\right\rfloor
    }
   (-1)^k{n + 1 \choose k}
   {n + b(j + 1) - 1 - i - d(n - 1) - bk\choose n}x^r.
 \end{eqnarray*}
 $j + 1 + \left\lfloor - \frac{1 + i + d(n - 1)}{b}\right\rfloor
 =
 j - \left\lfloor \frac{ i + d(n - 1)}{b}\right\rfloor,$
 $p_{i,j} = \frac{N}{b^n}$ and the theorem follows.
\end{proof}

\begin{lem}\label{lem:standardp}
 Let $P = (\tilde{p}_{i,j})$ be the matrix defined by $(\ref{eq:amazingmat})$ and
 let $p$ be defined by $(\ref{eq:mainth})$.
 Then,
 \begin{equation}\label{eq:pij}
 \tilde{p}_{i,j} =
 \frac{1}{b^n}\sum_{r = 0}^j
 (-1)^r
 {n + 1 \choose r}
 {{n + b(j - r) +\frac{b - 1}{p} - i}
 \choose
 n}.
 \end{equation}
\end{lem}
\begin{proof}
Let $s$ be the minimal element of the state
space  $\Omega_n(b,{\mathcal D})$ of the $n$-carry process.
Then, recall that
\[
 s = -\left\lceil(n - 1)\frac{-d}{b - 1}\right\rceil 
 = \left\lfloor(n - 1)\frac{d}{b - 1}\right\rfloor.
\]
Therefore,
 \[
 \tilde{p}_{i,j} =
 p_{i + s, j + s}
 =
  \frac{1}{b^n}
  \sum_{r \geq 0}
  (-1)^r{n + 1 \choose r}
  {n + b(j + s + 1 - r) - d(n - 1) - i - s - 1 \choose n}
 \]
 \begin{eqnarray*}
  & &
   n + b(j + s + 1 - r) - d(n - 1) - i - s - 1\\
  & = &
   n + (b - 1)s + b(j + 1 - r) - d(n - 1) - i  - 1\\
  & = &
   n + (b - 1)\left\lfloor(n - 1)\frac{-d}{b - 1}\right\rfloor
   + b(j + 1 - r) - d(n - 1) - i  - 1\\
  & = &
   n + (b - 1)
   \frac{(n - 1)d - (n - 1)d\pmod{b - 1}}{b - 1}
   + b(j + 1 - r) - d(n - 1) - i  - 1\\
  & = &
   n  + (b - 1) - (n - 1)d\pmod{b - 1}  + b(j - r) - i \\
  & = &
   n  + (b - 1)\frac{(b - 1) - (n - 1)d\pmod{b - 1}}{b - 1}  + b(j - r) - i \\
  & = &
   \begin{cases}
    n  + (b - 1)\left\{\frac{(n - 1)(-d)}{b - 1}\right\}  + b(j - r) - i 
    & \frac{(n - 1)(-d)}{b - 1}\in {\mathbb Z},
    \\
    n  + (b - 1) + b(j - r) - i 
    & \frac{(n - 1)(-d)}{b - 1}\not\in {\mathbb Z}.
   \end{cases}
 \end{eqnarray*}
 Here, $x\pmod{N}$ denotes the integer $y \in \{0,1,\ldots,N - 1\}$
 such that $x - y \in N{\mathbb Z}$.
 Then, we calculate a common upper bound of the range of the
 summation in $(\ref{eq:pij})$:
 \[
  n + b(j - r) + \frac{b -1}{p} - i \geq n
 \Longleftrightarrow
 r \leq j + \frac{b - 1}{pb} - \frac{i}{p}.
 \]
 Since $p\geq 1$ and $b > 1$,
 $j + \frac{b - 1}{pb} - \frac{i}{p} \leq j$.
\end{proof}

\subsection{Generalized Eulerian numbers}

\begin{lem}\label{lem:lastzero}
 \[
  v_{i,n + 1}^{(p)}(n) = 0.
 \]
\end{lem}
\begin{proof}
 \begin{eqnarray*}
  v_{i,n + 1}^{(p)}(n) & = &
   \sum_{r = 0}^{n + 1}(-1)^r{n + 1\choose r}[p(n + 1 - r) + 1]^{n - i}\\
 \end{eqnarray*}
 This is a linear combination of 
 \[
  \sum_{r = 0}^{n + 1}(-1)^{r}{n + 1\choose r}r^k = 0,\hspace{1cm}
 k = 0,1,\ldots, n.
 \]
\end{proof}
\begin{lem}\label{lem:recurrence}
\begin{equation}
 \label{eq:recurrence}
 v_{i,j}^{(p)}(n) = 
 [p(n + 1 - j) - 1]v_{i, j - 1}^{(p)}( n - 1 )
 +
 (pj + 1)v_{i,j}^{(p)}(n-1).
\end{equation}
\end{lem}
\begin{proof}
 The first term $T_1$ of right hand side of $(\ref{eq:recurrence})$ can
 be rewritten as
 \[
  T_1 = \sum_{k = 1}^j(-1)^{k-1}{n \choose k - 1}
 [p(n + 1 - j) - 1][p(j - k) + 1]^{n - 1 - i},
 \]
 and the second term $T_2$
 \[
  T_2 = \sum_{k = 1}^j(-1)^k{n \choose k}(pj + 1)[p(j - k) + 1]^{n - 1 - i}
 + (pj + 1)(pj + 1)^{n - 1 - i}.
 \]
 Thus the right hand side of $(\ref{eq:recurrence})$ is
 \begin{eqnarray*}
  T_1  + T_2 & = &
   \sum_{k = 1}^j(-1)^k
   \left\{
    -{n \choose k - 1}[p(n + 1 - j) - 1]
    +{n \choose k}(pj + 1)
   \right\}
   [p(j - k) + 1]^{n - 1 -i}\\
  &  &
   \hspace{3cm} +(pj + 1)^{n - i}\\
  & = &
   \sum_{k = 1}^j(-1)^k
   {n + 1 \choose k}[p(j - k) + 1]^{n - i}
   +   (pj + 1)^{n - i}\\
  & = & v_{i,j}^{(p)}(n).
 \end{eqnarray*}
\end{proof}

\begin{lem}
 \label{lem:sumeuler}
 \[
 \sum_{j = 0}^{n}v_{i,j}^{(p)}(n)
 =
 \begin{cases}
  p^{n}n! & \mbox{ if } i = 0,\\
  0 & \mbox{ if }i > 0.
 \end{cases}
 \]
\end{lem}
\begin{proof}
 By Lemma $\ref{lem:recurrence}$, we have
 \begin{eqnarray*}
  \sum_{j=0}^{n}v_{i, j}^{(p)}(n)
   & = &
   \sum_{j = 0}^{n}
   \left\{
    [p(n + 1 - j) - 1]v_{i, j - 1}^{(p)}(n - 1)
    +(pj + 1)v_{i, j}^{(p)}(n - 1)
   \right\}\\
   & = &
   \sum_{j = 0}^{n}
    [p(n + 1 - j) - 1]v_{i, j - 1}^{(p)}(n - 1)
    +\sum_{j = 0}^{n}
    (pj + 1)v_{i, j}^{(p)}(n - 1)
   \\
   & = &
   \sum_{j = 0}^{n - 1}
    [p(n - j ) - 1]v_{i, j }^{(p)}(n - 1)
    +
    \sum_{j = 0}^{n - 1}
    (pj + 1)v_{i, j}^{(p)}(n - 1)
   \\
   & = &
   \sum_{j = 0}^{n - 1}
    pm v_{i, j }^{(p)}(n - 1)
   \\
  & = &
   pn \sum_{j = 0}^{n - 1}v_{i, j }^{(p)}(n - 1)\\
   & = &
    \begin{cases}
     p^{n}n! & \mbox{ if } i = 0,\\
     0 & \mbox{ if }i > 0.
    \end{cases}
 \end{eqnarray*}
\end{proof}

The following Proposition $\ref{prop:symmm}$ shows a symmetry
of the generalized Eulerian number.
\begin{prop}\label{prop:symmm}
 Let $n$ be a positive integer. Then
 \[
  v^{(1)}_{i,n - 1 - j} = (-1)^iv^{(1)}_{i,j}(n)
 \mbox{~~~for }0\leq j\leq n - 1.
 \]
 Let $p > 1$ and $p^*$ be the real number satisfying
 \[
  \frac{1}{p} + \frac{1}{p^*} = 1.
 \]
 Then,
 \[
  v^{(p^*)}_{i,n - j}(n) = (-1)^i\left(\frac{p^*}{p}\right)^{n - i}v^{(p)}_{i,j}(n)
 \mbox{~~~for }0\leq j\leq n.
 \]
\end{prop}
\begin{proof}
 We show the proof only for the second part.
 The first part can be proved in the same manner.
 If $p > 1$ then $p^* = p/(p - 1)$.
 \begin{eqnarray*}
  v^{(p^*)}_{i, n - j}(n) & = &
   \sum_{k=0}^{n - j}(-1)^k{n + 1 \choose k}
   (p^*(n - j - k) + 1)^{n - i}\\
  & = & 
      \sum_{k = n + 1 - j}^{n + 1}(-1)^k{n + 1 \choose k}
   (p^*(n - j - k) + 1)^{n - i}\\
  & = & 
      - \sum_{k' = 0}^{j}(-1)^{n + 1 - k'}{n + 1 \choose n + 1 - k'}
   (p^*(n - j - (n + 1 - k')) + 1)^{n - i}\\
  & = & 
      - \sum_{k' = 0}^{j}(-1)^{n + 1 - k'}{n + 1 \choose k'}
   \left(\frac{p}{p-1}(k' - j - 1) + 1\right)^{n - i}\\
  & = & 
   - \sum_{k = 0}^{j}(-1)^{n + 1 - k}{n + 1 \choose k}
      \left(\frac{-1}{p - 1}\right)^{n - i}
   \left(p(j - k + 1) - p + 1\right)^{n - i}\\
  & = & 
   (-1)^{n}\left(\frac{-1}{p - 1}\right)^{n - i}
      \sum_{k = 0}^{j}(-1)^{k}{n + 1 \choose k}
   \left(p(j - k) + 1\right)^{n - i}\\
  & = & 
   \frac{(-1)^{i}}{(p - 1)^{n - i}}
   \sum_{k = 0}^{j}(-1)^{k}{n + 1 \choose k}
   \left(p(j - k) + 1\right)^{n - i}\\
 \end{eqnarray*}
\end{proof}
%

\subsection{Left eigenvectors}

\begin{proof}[Proof of Theorem $\ref{thm:main}$]
 The proof is essentially the same as that of Holte \cite{Holte}.
 It suffices to show that
 \[
  \sum_{k = 0}^{m - 1}v_{i,k}^{(p)}(n)\tilde{p}_{k,j} = \frac{1}{b^i}v_{i,j}^{(p)}(n).
 \]
 We prove the theorem for the case 
 in which $p\neq 1$, i.e., $m = n + 1$,
 and the other case can be proved in the same manner.
 By Lemma $\ref{lem:standardp}$ we have
 \[
  \tilde{p}_{k,j} =
   \frac{1}{b^n}
  \sum_{r = 0} ^j
  (-1)^r{n + 1 \choose r}
  {n + K(j,r) - k \choose n},
 \]
 where we put $K(j,r) = b(j - r) + \frac{b - 1}{p}$
 for the simplicity of the notation.

 \begin{eqnarray*}
  \sum_{k = 0}^{n}v_{i,k}^{(p)}(n)\tilde{p}_{k,j}
  & = &
   \sum_{k = 0}^n
   \frac{1}{b^n}
   \sum_{r = 0}^j
 (-1)^r{n + 1 \choose r}
 {n + K(j,r) - k \choose n}v_{i,k}^{(p)}(n)\\
  & = &
   \frac{1}{b^n}
   \sum_{r = 0}^{j}
   (-1)^r{n + 1 \choose r}
   \sum_{k = 0}^{K(j,r)}
  {n + K(j,r) - k \choose n}
  v_{i,k}^{(p)}(n)\\
  & = &
   \frac{1}{b^n}
   \sum_{r = 0}^{j}
   (-1)^r{n + 1 \choose r}
   \{pK(j,r) + 1\}^{n - i}.
 \end{eqnarray*}
 The third equality in the above transformation is derived as follows:
 First recall that $v_{i,k}^{(p)}(n)$ is the
 coefficient of $x^{k}$ in 
 \[
 \left(
 \sum_{\nu = 0}^{n + 1}(-1)^{\nu}
 {n + 1 \choose \nu} x^\nu
 \right)
 \left(
 \sum_{\mu = 0}^\infty (p\mu + 1)^{n - i}x^\mu
 \right)
 =
 (1 - x)^{n + 1}
 \left(
 \sum_{\mu = 0}^\infty (p\mu + 1)^{n - i}x^\mu
 \right),
 \]
 and 
 \[
  \frac{1}{(1 - x)^{n + 1}}
 =
 \sum_{k = 0}^\infty
 {n + k \choose n}x^k.
 \]
 Therefore, $\sum_{k = 0}^{K(j,r)}{n + K(j,r) - k \choose n} v_{i,k}^{(p)}(n)$
 is the coefficient of $x^{K(j,r)}$ in
 $\sum_{\mu = 0}^\infty (p\mu + 1)^{n - i}x^\mu$.

 It can be easily confirmed that $pK(j,r) + 1 = b(p(j - r) + 1)$
 which completes the proof.
\end{proof}

Theorem $\ref{thm:main}$ gives a way of finding a numeration
system $(b,{\mathcal D})$ whose $n$-carry process
has the stationary distribution of the form
\[
 \pi = (\pi(s), \pi(s + 1), \ldots ,\pi(s + m - 1))
 =
 \frac{1}{p^nn!}
 \left(\eunump{n}{0},\eunump{n}{1},\ldots,\eunump{n}{m - 1}\right)
\]
for any given $n$ and rational $p\geq 1$.
For instance, if $p = \frac{K}{L}$ where 
$K$ and $L$ are coprime positive integers such that $K\geq L$,
then we can choose $b$ and $d$ as
\begin{equation}\label{eq:findbd}
 b = (n - 1)K + 1,~~~~~ d = -L. 
\end{equation}

\begin{exm}
 We construct numeration systems for $p = 2$ and $5/3$.

\begin{table}[H]
 \begin{center}
  \[
   \begin{array}{c|c|c|c|c}
    n & b & {\mathcal D} & P & V\\
    \hline
    & & & & \\
     2 & 3 & \{-1, 0, 1\} & 
     \frac{1}{3^2}     
     \begin{pmatrix}
      3 & 6 & 0\\
      1 & 7 & 1\\
      0 & 6 & 3
     \end{pmatrix}
     &
     \begin{pmatrix}
      1 &  6 & 1\\
      1 & 0 & -1\\
      1 & -2 & 1
     \end{pmatrix}
     \\
    & & & & \\
    3 & 5 & \{-1, 0, 1, 2, 3\} &
    \frac{1}{5^3}
    \begin{pmatrix}
     10 & 80 & 35 & 0\\
     4  & 68 & 52 & 1\\
     1  & 52 & 68 & 4\\
     0  & 35 & 80 & 10
    \end{pmatrix}
    &
    \begin{pmatrix}
     1 & 23 & 23 & 1\\
     1 & 5 & -5 & -1\\
     1 & -1 & -1 & 1\\
     1 &  -3 & 3 & -1
    \end{pmatrix}
    \\
    &&&&\\
    4 & 7 &\{-1, 0, 1, 2, 3, 4, 5\} &
    \frac{1}{7^4}
     \begin{pmatrix}
      35 & 826 & 1330 & 210 & 0 \\
      15 & 640 & 1420 & 325 & 1 \\
      5  & 470 & 1451 & 470 & 5 \\
      1  & 325 & 1420 & 640 & 15\\
      0  & 210 & 1330 & 826 & 35
     \end{pmatrix}
     &
    \begin{pmatrix}
     1 & 76 & 230 & 76 & 1\\
     1 & 22 & 0 & -22 & -1\\
     1 & 4 & -10 & 4 & 1\\
     1 & -2 & 0 & 2 & -1\\
     1 & -4 & 6 & -4 & 1
    \end{pmatrix}
   \end{array}
  \]
  \caption{Amazing matrices with $p = 2$}
 \end{center}
\end{table}

\begin{table}[H]
 \begin{center}
  \[
   \begin{array}{c|c|c|c|c}
    n & b & {\mathcal D} & P & V\\
    \hline
    & & & & \\
    2 & 6 & \{-3,-2,\ldots,2\}&
     \frac{1}{6^2}
    \begin{pmatrix}
     10 & 25 & 1\\
     6 & 27 & 3\\
     3 & 27 & 6
    \end{pmatrix}
    &
    \begin{pmatrix}
     1&  37/9 &4/9\\
     1&  -1/3 &-2/3\\
     1&  -2   &1
    \end{pmatrix}
    \\
    %
    %
    & & & & \\
    3 & 11 & \{-3,-2,\ldots,7\} &
     \frac{1}{11^4}
     \begin{pmatrix}
      84 & 804 & 439 & 4\\
      56 & 745 & 520 & 10\\
      35 & 676 & 600 & 20\\
      20 & 600 & 676 & 35
     \end{pmatrix}
     &
     \begin{pmatrix}
      1 & 404/27 &311/27 & 8/27\\
      1 & 28/9 & -11/3 & -4/9\\
      1 & -4/3 & -1/3 & 2/3\\
      1 & -3 & 3 & -1
     \end{pmatrix}
     \\
    %
    %
    & & & & \\
    4 & 16 & \{-3,-2,\ldots,12\} &
     {\small
     \frac{1}{16^4}
     \begin{pmatrix}
      715 & 20176 & 37390 & 7240 & 15\\
      495 & 18000 & 38326 & 8680 & 35\\
      330 & 15900 & 38960 & 10276&  70\\
      210 & 13900 & 39280 & 12020&  126\\
      126 & 12020 & 39280 & 13900&  210
     \end{pmatrix}
     }
     &
     {\small 
     \begin{pmatrix}
      1 & 3691/81 & 8891/81 & 2321/81 & 16/81\\
      1 & 377/27  & -31/9   & -101/9  & -8/27\\
      1 & 19/9    & -61/9   & 29/9    & 4/9\\
      1 & -7/3    & 1       & 1       & -2/3\\
      1 & -4      & 6       & -4      & 1
     \end{pmatrix}
     }
   \end{array}
  \]
  \caption{Amazing matrices with $p = \frac{5}{3}$}
 \end{center}
\end{table}
%
%

\end{exm}

\subsection{Sum of independent uniform random variables}

\begin{proof}[Proof of Corollary $\ref{cor:sumeuler}$]
Let $X_{i,j} (i = 1, 2, \ldots, n, j = 1,2,\ldots)$ be
independent random variables each distributed uniformly
on ${\mathcal D}$.
Then, for each integer $k\geq 1$, the random variables
\[
 X_{i}^{(k)} = 
 \frac{X_{i,1}}{b} +  \frac{X_{i,2}}{b^2} + 
 \cdots
 + \frac{X_{i,k}}{b^k}, ~~~~~i =1,2,\ldots,n
\]
are independent random variables uniformly distributed over the
set
\[
 R_k = \left\{
 \left.
 \frac{x_{1}}{b} +  \frac{x_{2}}{b^2} + 
 \cdots
 + \frac{x_{k}}{b^k}\,\right|\, x_{i}\in {\mathcal D}
 \right\}.
\]
Therefore
\[
 \lim_{k\rightarrow\infty}\pr(X_i^{(k)} \in [a,b])
 = b - a, ~~~\mbox{ for } a \geq b\in [l, l+1] \mbox{ and } i=1,2,\ldots,n.
\]
Let $X_1, X_2, \ldots, X_n$ be independent random variables
each of which is distributed uniformly over $[l, l + 1]$.
Then, for any integer $c\in\Omega$,
\[
 \lim_{k\rightarrow\infty}\pr(X_1^{(k)} + X_2^{(k)} + \cdots + X_n^{(k)} \in c + [l, l + 1])
 = \pr(X_1 + X_2 + \cdots X_n  \in c + [l, l + 1]).
\]
 Since
\[
  \lim_{k\rightarrow\infty}\pr(X_1^{(k)} + X_2^{(k)} + \cdots + X_n^{(k)} \in c + [l, l + 1])
 =
 \pi(c),
\]
 we have
 \[
  \pi(c) = \pr(X_1 + X_2 + \cdots X_n  \in c + [l, l + 1]).
 \]
 Let $p > 1$ be a non-integral rational number, and
 suppose that we choose $b$ and $d$ so that
 \[
 p = \frac{1}{\{(n - 1)(-l)\}}
 \]
 holds, which is always possible by $(\ref{eq:findbd})$. Then, we have
 \begin{eqnarray*}
   \frac{1}{p} = (n - 1)(-l) - \lfloor(n - 1)(-l)\rfloor
  &\Longleftrightarrow &
   nl + \frac{1}{p} = l - \left\lfloor(n - 1)(-l)\right\rfloor 
   = l + (1 - \lceil(n - 1)(-l)\rceil)\\
  &\Longleftrightarrow &
   nl + \frac{1}{p} + [k - 1, k] = (s + k) + [l, l + 1].
 \end{eqnarray*}
 Let $Y_1, Y_2, \ldots, Y_n$ be independent random variables
 each uniformly distributed on the unit interval $[0,1]$.
 \begin{eqnarray*}
  \pr(Y_1 +  Y_2 + \ldots + Y_n \in \frac{1}{p} + [k - 1, k])
   & = &
  \pr((Y_1 + l) +  \ldots + (Y_n + l) \in nl + \frac{1}{p} + [k - 1, k])\\
  & = &
   \pr(X_1 + X_2 + \cdots + X_l \in nl + \frac{1}{p} + [k - 1, k])\\
  & = &
   \pr(X_1 + X_2 + \cdots + X_l \in (s + k) + [l, l+1])\\
  & = &
   \pi(s + k)\\
   & = &
   \frac{1}{p^nn!}\eunump{n}{k}.
 \end{eqnarray*}
 When $p$ is an integer, the statement is proved by a similar argument. 
 Since both sides of the equation $(\ref{eq:sumunifiid})$ are
 continuous function of $p$,
 the statement of the theorem holds when $p$ is irrational.
\end{proof}

\begin{rmk}
 Corollary $\ref{cor:sumeuler}$ gives a different proof and
 a new interpretation for {\rm Proposition} $\ref{prop:symmm}$ for the case $i = 0$.
\end{rmk}

\section{Negative base}

In this section, we consider carries processes over
the numeration systems with the negative bases.
Let $b > 1$ be an integer and $\mathcal{D} = \{d, d + 1, \ldots, d + b - 1\}$ 
a set of integers containing $0$.
Suppose an integer $x$ can be represented in the form:
\[
 x = (x_l x_{l - 1} \cdots x_{0})_{-b}
 = x_l(-b)^l + x_{l - 1}(-b)^{l - 1} + \cdots + x_1(-b) + x_0,
\]
where $l$ is a non-negative integer and $x_l\neq 0$.
Then this representation is unique and the set
\[
 \left\{(x_lx_{l - 1} \ldots x_{0})_{-b}\,|\, l\geq 0, x_k \in {\mathcal D}\right\}.
\]
is closed under the addition.
We can define the $n$-carry process  over the
numeration system $(-b, {\mathcal D})$ in the same 
manner as the positive base case. Let $\{X_{i,j}\}_{1\leq i\leq n, j\geq 0}$
be a set of i.i.d. random variables each distributed uniformly on ${\mathcal D}$.
Then the carries process $(C_0, C_1, C_2, \ldots)$ is defined  as follows:
$\pr(C_0 = 0) = 1$ and
\[
 C_i = \frac{C_{i - 1} + X_{1,i-1} + \cdots + X_{n,i-1} - A_{i-1}}{-b}~~~~
 \mbox{ for }i > 0,
\]
where $A_j$ is ${\mathcal D}$-valued.
These carries processes have the properties similar to those
of the positive base cases.
The proofs of the following Lemma $\ref{lem:negstatespace}$,
Theorem $\ref{thm:negtransp}$, and Lemma $\ref{lem:negstandardp}$ are similar to those of
Lemma $\ref{lem:statespace}$,Theorem $\ref{thm:transp}$, and Lemma $\ref{lem:standardp}$.
The proof of Theorem $\ref{thm:negamazing}$ needs
an additional combinatorial argument.

\begin{lem}\label{lem:negstatespace}
  Let $\Omega  = \Omega_n(-b,{\mathcal D})$ be the state space of the
  $n$-carry process over the numeration system 
  $(-b,{\mathcal D})$.
  Then, $\Omega = \{s,s+1, \ldots,t\}$ with
  \[
   s = -\lceil(n - 1)(- l)\rceil = \lfloor(n - 1)l\rfloor,
   ~~
   t = \lceil (n - 1)(l + 1) \rceil,
  \]
  where $l = (-d-b)/(b + 1)$.
  Therefore, the size of the state space $\Omega$ is
  \[
   \#\Omega =
  \begin{cases}
   n + 1 & (n - 1)l \not\in {\mathbb Z},\\
   n &  (n - 1)l \in {\mathbb Z}.
  \end{cases}
  \]
 \end{lem}
\begin{thm}\label{thm:negtransp}
 Let $i,j\in \Omega(-b,{\mathcal D})$. Then the transition probability
 $p_{i,j} = \pr(C_{t+1} = j \,|\, C_{t} = i)$ for $t > 0$ is
 \[
  p_{i,j} =
 \frac{1}{b^n} \sum_{r = 0}^{-j + 1 + \left\lfloor\frac{-i-1+(1-n)d}{b}\right\rfloor}
 (-1)^r{n + 1\choose r}{n - b(j-1+r)-i-1+(1-n)d \choose n}.
 \]
\end{thm}
We denote
\[
 \tilde{p}_{i,j}=p_{i + s, j + s}.
\]
where $s$ is the minimal element of the
state space $\Omega(-b, {\mathcal D})$ of the $n$-carry process over $(-b,{\mathcal D})$.
\begin{lem}\label{lem:negstandardp}
 Let $p$ be defined by 
 \begin{equation}\label{eq:negmain}
  p = 
 \begin{cases}
   \frac{1}{\{(n - 1)l\}} & (n - 1)l \not\in {\mathbb Z} 
  ~~(\Leftrightarrow m = n + 1),\\
  1 &  (n - 1)l \in {\mathbb Z}
  ~~(\Leftrightarrow m = n).
 \end{cases}
 \end{equation}
 where $l = (-d - b)/(b + 1)$ and $m = \#\Omega_n(-b,{\mathcal D})$.
 Then,
 \[
 \tilde{p}_{i,j} = \frac{1}{b^n}
 \sum_{r = 0}^{n - j}(-1)^r
 {n + 1 \choose r}
 {n + b(n + 1 - j - r) - \frac{b + 1}{p} - i
 \choose
 n
 }.
 \]
\end{lem}
\begin{thm}\label{thm:negamazing}
 Let $P=(\tilde{p}_{i,j})_{0\leq i,j\leq m-1}$ be the transition probability matrix
 of the $n$-carry process over the numeration system
 $(-b, {\mathcal D})$,and $m = \#\Omega(-b,{\mathcal D})$
 be the size of the state space.
 Let $p$ be defined by 
 \[
  p = 
 \begin{cases}
   \frac{1}{\{(n - 1)l\}} & (n - 1)l \not\in {\mathbb Z} 
  ~~(\Leftrightarrow m = n + 1),\\
  1 &  (n - 1)l \in {\mathbb Z}
  ~~(\Leftrightarrow m = n).
 \end{cases}
 \]
 where $l = (-d - b)/(b + 1)$ and $\{x\} = x - \lfloor x\rfloor$.
 Let $V= (v_{i,j}^{(p)})_{0\leq i,j\leq m -1}$.
 Then, we have
 \[
 VPV^{-1} = {\rm diag}\left(1, (-b)^{-1}, \ldots, (-b)^{m - 1}\right).
 \]
 In particular, the $n$-carry process over $(-b,{\mathcal D})$ has
 the stationary distribution
 \[
  \pi = \left(\pi(s), \pi(s + 1), \ldots, \pi(s + m - 1)\right)
 =
  \frac{1}{p^nn!}
 \left(\eunump{n}{0},\eunump{n}{1},\ldots,\eunump{n}{m - 1}\right).
 \]
\end{thm}
\begin{proof}
 It suffices to show that
 \begin{equation}\label{eq:negeigen}
 \sum_{k=0}^{m - 1}v_{i,k}^{(p)}(n)\tilde{p}_{k,j}
 =
  \frac{1}{(-b)^i}v_{i,j}^{(p)}(n).
 \end{equation}
 Recall that
 \[
 \tilde{p}_{k,j}
 =
 \frac{1}{b^n}
 \sum_{r = 0}^{n-j}
 (-1)^r
 {n + 1 \choose r}
 {n + K(j,r) - k \choose n},
 \]
 where, we put $K(j,r) = b(n + 1 - j - r) -\frac{b + 1}{p}$ for the simplicity
 of the notation.
 Therefore,
 \begin{eqnarray*}
  \mbox{L.H.S. of }(\ref{eq:negeigen})  & = &
   \sum_{k = 0}^m
   \frac{1}{b^n}\sum_{r = 0}^{n - j}
   (-1)^r{n + 1 \choose r}{n + K(j,r) - k \choose n}v_{i,k}^{(p)}(n)\\
  & = &
   \frac{1}{b^n}\sum_{r = 0}^{n - j}
   \sum_{k = 0}^{K(j,r)}
   (-1)^r{n + 1 \choose r}{n + K(j,r) - k \choose n}v_{i,k}^{(p)}(n)\\
  & = &
   \frac{1}{b^n}\sum_{r = 0}^{n - j}
   (-1)^r{n + 1 \choose r}
   \sum_{k = 0}^{K(j,r)}{n + K(j,r) - k \choose n}v_{i,k}^{(p)}(n)\\
  & = &
   \frac{1}{b^n}\sum_{r = 0}^{n - j}
   (-1)^r{n + 1 \choose r}
   \left[pK(j,r) + 1\right]^{n - i}\\
  &  &~~~~~~~~~~~~~~~~~~
   \left(\mbox{We use the same argument as in the proof of Theorem }\ref{thm:transp}.\right)
   \\
  & = &
   \frac{1}{b^n}\sum_{r' = j + 1}^{n + 1}
   (-1)^{n + 1 - r'}{n + 1 \choose n + 1 - r'}
   \left[pK(j,n + 1 - r') + 1\right]^{n - i}\\
  &  &~~~~~~~~~~~~~~~~~~
   \left(\mbox{We use the transformation } r'= n + 1 -r.\right)
   \\
  & = &
   \frac{1}{b^n}\sum_{r = j + 1}^{n + 1}
   (-1)^{n + 1 - r}{n + 1 \choose r}
   \left[-b(p(j - r') + 1)\right]^{n - i}.
 \end{eqnarray*}
\end{proof}

\section{Concluding remarks}

Many natural questions arise.

In the forthcoming paper, we will show a 
formula for  the right eigenvectors,
which involves Stirling numbers.

Our theorems hold only for the numeration systems
$(b,{\mathcal D})$, where ${\mathcal D}$ consists of
consecutive integers containing $0$.
For example, the $2$-carry process over $(3,\left\{-1, 0, 4\right\})$
has rather large state space
\[
 \Omega_2(3,\{-1,0,4\}) = \{-5,-4,\ldots,4\},
\]
and the transition probability matrix
\[
P =\frac{1}{9}
\begin{pmatrix}
1 & 2 & 0 & 3 & 0 & 2 & 1 & 0 & 0 & 0\\
2 & 0 & 0 & 2 & 1 & 4 & 0 & 0 & 0 & 0\\
1 & 0 & 2 & 0 & 3 & 2 & 0 & 1 & 0 & 0\\
0 & 1 & 2 & 0 & 3 & 0 & 2 & 1 & 0 & 0\\
0 & 2 & 0 & 0 & 2 & 1 & 4 & 0 & 0 & 0\\
0 & 1 & 0 & 2 & 0 & 3 & 2 & 0 & 1 & 0\\
0 & 0 & 1 & 2 & 0 & 3 & 0 & 2 & 1 & 0\\
0 & 0 & 2 & 0 & 0 & 2 & 1 & 4 & 0 & 0\\
0 & 0 & 1 & 0 & 2 & 0 & 3 & 2 & 0 & 1\\
0 & 0 & 0 & 1 & 2 & 0 & 3 & 0 & 2 & 1
\end{pmatrix},
\]
whose characteristic polynomial $\det(xI - P)$ is
\[
 (x - 1)
 (3x - 1)
 (9x - 1)
 (531441x^7 - 19683x^5 + 5103x^4 - 1944x^3 - 297x^2 + 24x + 2).
\]
Although there are eigenvalues of the form $1,1/3,1/3^2$,
we have no knowledge on the rest of the eigenvalues.
The difficulty comes from the geometric structure of the fundamental
domain \\$ \left\{(x_lx_{l - 1} \ldots x_{0})_{b}\,|\, l\geq 0, x_k \in {\mathcal D}\right\}.$

Diaconis and Fulman \cite{DiaconisFulman,DiaconisFulman2} shows the
relation between carries processes and shufflings for the
case when $p = 1$ and $2$.
We do not know whether there exist some shufflings
corresponding to the cases with $p \neq 1, 2$.


\begin{thebibliography}{1}

\bibitem{ChowGessel}
C.~Chow, I.~M. Gessel,
\newblock On the descent numbers and major indices for the hyperoctahedral
  group,
\newblock { Adv. Appl. Math.}, 38 (2007) 275--301. 

\bibitem{DiaconisFulman}
P.~Diaconis, J.~Fulman,
\newblock Carries, shuffling, and an amazing matrix,
\newblock {Amer. Math. Monthly}, 116 (2009) 780--803.

\bibitem{DiaconisFulman2}
P.~Diaconis, J.~Fulman,
\newblock Carries, shuffling, and symmetric functions,
\newblock {Adv. App. Math.}, 43 (2009) 176--196.

\bibitem{Feller}
W.~Feller,
\newblock {An Introduction to Probability Theory and its
  Applications, volume II, 2nd edition},
\newblock John Wiley \& Sons, Inc., New York-London-Sydney 1971.

\bibitem{Holte}
J.~Holte,
\newblock Carries, combinatorics, and an amazing matrix,
\newblock {Amer. Math. Monthly}, 104 (1997) 138--149.

\bibitem{MacMahon}
P.~A. MacMahon,
\newblock The divisors of numbers,
\newblock {Proc. London Math. Soc.}, 19 (1921) 305--340.

\bibitem{NovelliThibon}
J.-C. Novelli, J.-Y. Thibon,
\newblock Noncommutative symmetric functions and an amazing matrix,
\newblock {Adv. App. Math.}, 48 (2012) 528--534.


\bibitem{oeismacmahon}
N.~J.~A. Sloane,
\newblock The {O}n-{L}ine {E}ncyclopedia of {I}nteger {S}equences.
\newblock {http://oeis.org/A060187}.

\end{thebibliography}

\end{document}